\documentclass[10pt]{article}
\usepackage[utf8]{inputenc}
\usepackage[a4paper, margin=1in]{geometry}
\usepackage{amsmath}
\usepackage{array}
\newcolumntype{P}[1]{>{\centering\arraybackslash}p{#1}}
\newcolumntype{M}[1]{>{\centering\arraybackslash}m{#1}}
\usepackage{xcolor}
\usepackage{amsfonts, amssymb}
\usepackage{amsthm}
\usepackage{esvect}
\usepackage[english]{babel}
\usepackage[ruled,linesnumbered]{algorithm2e}
\usepackage[T1]{fontenc}
\usepackage{multirow}
\usepackage{graphicx}
\usepackage{cite}
\usepackage[numbers,sort]{natbib}
 \usepackage{titlesec}
\titleformat{\section}
  {\normalfont\fontfamily{ptm}\fontsize{11}{11}\bfseries}{\thesection}{1em}{}
\titleformat{\subsection}
  {\normalfont\fontfamily{ptm}\fontsize{10}{11}\bfseries}{\thesubsection}{1em}{}
\titleformat{\subsubsection}
  {\normalfont\fontfamily{ptm}\fontsize{10}{11}\selectfont}{\thesubsubsection}{1em}{}
\newtheorem{theorem}{Theorem}[section]

\newtheorem{lemma}[theorem]{Lemma}
\newtheorem{proposition}[theorem]{Proposition}
\newtheorem{definition}[theorem]{Definition}
\newtheorem{example}[theorem]{Example}
\usepackage{hyperref}
\hypersetup{
    colorlinks=true,    
    urlcolor=cyan,
    linkcolor=blue,
    citecolor=blue
}
\usepackage{mathtools}

\DeclarePairedDelimiter\floor{\lfloor}{\rfloor}
\usepackage{soul}

\providecommand{\keywords}[1]
{
  \small	
  \textbf{\textit{Keywords---}} #1
}
\newcommand{\revised}[1]{{#1}}

\newcommand{\ignore}[1]{}

\title{\large A Polyhedral Approach to Bisubmodular Function Minimization}

\author{ \small Qimeng Yu\thanks{Department of Industrial Engineering and Management Sciences, Northwestern University, Evanston, IL, USA \{kim.yu@northwestern.edu\}} \quad Simge K\"u\c{c}\"ukyavuz\thanks{Department of Industrial Engineering and Management Sciences, Northwestern University, Evanston, IL, USA \{simge@northwestern.edu\}} }
\date{\small \today} 
\begin{document}
\maketitle

\begin{abstract}
\noindent 
We consider minimization problems with bisubmodular objective functions. We propose valid inequalities, namely the poly-bimatroid inequalities, and provide a complete linear description of the convex hull of the epigraph of a bisubmodular function. Furthermore, we develop a cutting plane algorithm for constrained bisubmodular minimization based on the poly-bimatroid inequalities. Our computational experiments on \revised{the minimization subproblem in} robust coupled sensor placement show that our algorithm can solve highly non-linear problems that do not admit compact mixed-integer linear formulations. 
\end{abstract}


\section{Introduction}
\label{sect:intro}
\noindent Bisubmodularity---first considered in \cite{chandrasekaran1988pseudomatroids,qi1988directed}---is a natural extension of submodularity to set functions with two arguments. Next we give a formal definition of bisubmodularity.

\noindent Let $N=\{1,2,\dots,n\}$ be a finite non-empty set, and let $3^N = \{(S_1,S_2)\mid S_1, S_2 \subseteq N, S_1\cap S_2 = \emptyset\}$ denote the collection of all pairs of disjoint subsets of $N$. A function $f:3^N \rightarrow \mathbb{R}$ is \textit{bisubmodular} if 
\begin{align*}
f(X_1, X_2) + f(Y_1, Y_2) \geq f(X_1\cap Y_1, X_2 \cap Y_2) + f((X_1\cup Y_1 )\backslash (X_2 \cup Y_2), (X_2 \cup Y_2)\backslash (X_1\cup Y_1 ))
\end{align*} for any $(X_1, X_2), (Y_1, Y_2)\in 3^N$. Without loss of generality, we assume that $f(\emptyset,\emptyset) = 0$. By slightly abusing notation, for any $x\in \{\pm1,0\}^n$, $f(x)$ is equivalent to $f(S_1^x,S_2^x)$, or simply $f(S_1,S_2)$, where $S_1 = \{i\in N\mid x_i = 1\}$ and $S_2 = \{i\in N \mid x_i = -1\}$. Similarly, for any $(S_1, S_2)\in 3^N$, we let $x_{S_1,S_2}$ be the corresponding ternary characteristic vector. A partition of $N$ is any $(S,T)$ such that $S\cup T = N$ and $S\cap T = \emptyset$. The function $f$ is said to be \textit{bisubmodular over a partition} $(S,T)$ if $f'(X) := f(X\cap S, X\cap T)$ is submodular over $X\subseteq N$.

\noindent  \citet{ando1996characterization} provide an alternative definition of bisubmodularity. A function $f:3^N \rightarrow \mathbb{R}$ is bisubmodular if and only if 
\begin{enumerate}
\item[(A1)] the function $f$ is bisubmodular over every partition of $N$, and 
\item[(A2)] for any $(S_1, S_2)\in 3^N$ and $i\not \in S_2\cup S_2$, $f(S_1\cup\{i\}, S_2) + f(S_1,S_2\cup\{i\}) \geq 2f(S_1,S_2)$.
\end{enumerate} In later sections, we refer to these two conditions as the Ando Conditions. 

\noindent There has been growing interest in bisubmodular minimization problems of the form
\begin{equation}
\label{eq:1}
\min_{(S_1,S_2)\in 3^N} f(S_1,S_2),
\end{equation} 
where $f$ is a bisubmodular function defined over a base set $N$. \citet{qi1988directed} generalizes Lov\'asz's extension to bisubmodular functions, which suggests that bisubmodular minimization problems are polynomially solvable using the ellipsoid method. Other researchers take algorithmic approaches to tackle the \textit{unconstrained} bisubmodular minimization problems. \citet{fujishige2005bisubmodular} and \citet{mccormick2010strongly}  propose a weakly and a strongly polynomial-time bisubmodular minimization algorithm, respectively. Both algorithms are based on a min-max theorem proposed by \citet{fujishige1997min}, which establishes the equivalence of Problem \eqref{eq:1} and an $\mathcal{L}^1$-norm maximization problem over a bisubmodular polyhedron defined in Section \ref{sect:poly_ineq}.

\noindent It is known that simple constraints, such as cardinality constraints, make submodular minimization NP-hard \cite{svitkina2011submodular}. Since submodular minimization can be reduced from bisubmodular minimization, constrained bisubmodular minimization problems are also NP-hard. In contrast to the aforementioned methods for unconstrained variants, we pursue a polyhedral approach to bisubmodular minimization so that constrained variants of this problem can be solved exactly via a cutting plane method. 
 
\noindent Consider the convex hull of the epigraph of a bisubmodular function $f$, 
\[\mathcal{Q}_f = \text{conv}\{(x,z)\in \{\pm 1, 0\}^n\times \mathbb{R} \mid f(x) \leq z\}. \] 
We can equivalently state Problem (\ref{eq:1}) as
\begin{equation}
\label{eq:2}
\min\{z \mid (x,z)\in \mathcal{Q}_f\}.
\end{equation}

\noindent In our approach, we introduce a class of valid inequalities for $\mathcal{Q}_f$, which we refer to as the \textit{poly-bimatroid inequalities}. We prove that these inequalities, along with trivial bound constraints, fully describe $\mathcal{Q}_f$. Using the proposed inequalities, we propose an exact cutting plane method that solves the constrained bisubmodular minimization problems. Our computational experiments motivated by a robust coupled sensor placement problem show that our cutting plane algorithm can handle challenging constrained bisubmodular minimization problems that cannot be formulated as compact mixed-integer linear programs. 

\noindent The outline of this paper is as follows. In Section \ref{sect:poly_ineq}, we review the preliminaries of  bisubmodular polyhedra. In Section \ref{sect:convex_hull}, we propose the poly-bimatroid inequalities and establish the complete linear  description of $\mathcal{Q}_f$ using the proposed inequalities. In Section \ref{sect:cut_plane}, we develop a cutting plane algorithm for general bisubmodular minimization problems. In  Section \ref{sect:computation}, we introduce an application, namely  a robust coupled sensor placement problem, which involves a bisubmodular minimization subproblem that is  highly nonlinear. 
In Section \ref{sect:exp}, we summarize our computational experiments on this subproblem. Lastly, we include a few concluding remarks in Section \ref{sect:conclusion}. 

\section{Preliminaries}
\label{sect:poly_ineq}
\noindent In this section, we review the preliminaries of bisubmodularity. Recall that $3^N$ denotes the collection of all pairs of disjoint subsets of $N$. Given a bisubmodular function $f:3^N \rightarrow \mathbb{R}$ with $f(\emptyset, \emptyset)=0$, the set 
\[ \mathcal{P}_f = \{ \pi\in\mathbb{R}^n \mid \sum_{i\in S_1} \pi_i - \sum_{j\in S_2}\pi_j \leq f(S_1, S_2), \forall (S_1, S_2)\in 3^N\}\]  is a \textit{bisubmodular polyhedron} associated with the bisubmodular system $(3^N, f)$. Such a polyhedron is introduced by \citet{chandrasekaran1988pseudomatroids} under the name "pseudomatroid." Other researchers have also considered this concept and coined different names, such as "$\Delta$-matroid"  \cite{bouchet1995delta} and "ditroid"  \cite{qi1988directed}.   
 
\noindent \citet{ando1996structures} show that the bisubmodular polyhedron $\mathcal{P}_f$ is bounded if and only if the associated family of subsets of $N$ is $3^N$. Furthermore, the authors show that the extreme points of $\mathcal{P}_f$ can be generated by the \textit{signed greedy algorithm}  \cite{mccormick2010strongly} described in
Algorithm \ref{signedgreedy}. 

\begin{algorithm}[H]
\label{signedgreedy}
\SetAlgoLined
\textbf{Input} a permutation of $N$, namely $\delta = \{\delta_1, \delta_2, \dots, \delta_n\}$ and a sign vector $\sigma\in\{\pm 1\}^N$\;
$\pi \leftarrow \mathbf{0}$, $S_1 \leftarrow \emptyset$, $S_2 \leftarrow \emptyset$\;
\For{$i=1,2,\dots,n$}{
    \If{$\sigma_{\delta_i} = 1$}{
    $\pi_{\delta_i} \leftarrow f(S_1\cup\{\delta_i\}, S_2) - f(S_1, S_2)$\;
    $S_1 \leftarrow S_1\cup \{\delta_i\}$\;
    }
    \Else{
    $\pi_{\delta_i} \leftarrow -f(S_1, S_2\cup\{\delta_i\}) + f(S_1, S_2)$\;
    $S_2 \leftarrow S_2\cup \{\delta_i\}$\;
    }
}
   \textbf{Output} An extreme point $\pi\in\mathcal{P}_f$. 
   \caption{Signed Greedy \cite{ando1996structures, mccormick2010strongly}}
\end{algorithm} 

\vspace{0.3cm}  \noindent In this algorithm, we start with an ordering $\delta$ of $N$ and a sign vector $\sigma\in\{\pm 1\}^N$. The output is an extreme point $\pi \in\mathcal{P}_f$ that is consistent with $\delta$ and $\sigma$. In other words, for every $i\in N$, if $\sigma_i$ is 1, then $i$ is placed in the first argument of $f$. Otherwise it is included in the second argument of $f$. In either case, the $i$th entry of $\pi$ is the marginal change of the function value by appending $i$ to the chosen argument of $f$, as captured by lines 5-6 and 9-10 of Algorithm \ref{signedgreedy}.  

\noindent Furthermore, researchers have studied maximizing linear objectives over the bisubmodular polyhedra. Given any $\overline{x}\in\mathbb{R}^n$, such a maximization problem can be written as
\begin{equation}
\label{eq:linear_bisub_max}
\max\{\pi^\top \overline{x} \mid \pi \in \mathcal{P}_f\}.
\end{equation} 
Problems of this form occur as separation problems in a cutting-plane framework, which we will discuss in Section \ref{sect:cut_plane}. \citet{bouchet1987greedy} gives an $\mathcal{O}(n\log n)$ greedy algorithm to solve problem \eqref{eq:linear_bisub_max}. 
\begin{proposition} (\citet{bouchet1987greedy})
\label{prop:2}
Algorithm \ref{alg:sepa} determines an optimal $\pi$ for \eqref{eq:linear_bisub_max}. 
\end{proposition}

\begin{algorithm}[H]
 \caption{Generalized Greedy \cite{bouchet1987greedy}}
\label{alg:sepa}
\SetAlgoLined
\textbf{Input}  $\overline{x}\in \mathbb{R}^n$\;
sort entries in $\overline{x}$ such that $|\overline{x}_{\delta_1}| \geq |\overline{x}_{\delta_2}| \geq \dots \geq |\overline{x}_{\delta_n}|$\;
$\pi \leftarrow \mathbf{0}$,
$S_1, S_2 \leftarrow \emptyset$\;
\For{$i = 1, 2, \dots, n$}{
   \If{$\overline{x}_{\delta_i}\geq 0$}{
    $\pi_{\delta_i} \leftarrow f(S_1\cup\{\delta_i\}, S_2) - f(S_1, S_2)$\;
    $S_1 \leftarrow S_1\cup \{\delta_i\}$\;
    }
    \Else{$\pi_{\delta_i} \leftarrow -  f(S_1, S_2\cup\{\delta_i\}) + f(S_1, S_2)$\;
     $S_2 \leftarrow S_2\cup \{\delta_i\}$\;}
}
   \textbf{Output} $\pi$.
\end{algorithm}

\vspace{0.3cm} \noindent Algorithm \ref{alg:sepa} first sorts the absolute values of entries in $\overline{x}$ so that they are non-increasing. Let this ordering of $N$ be $\delta$. Lines 5-11 assign $\delta_i \in N$ to the first argument of $f$ if $\overline{x}_{\delta(i)}\geq 0$ and to the second argument otherwise. In other words, we obtain a sign vector $\sigma$ based on the signs of every entry in $\overline{x}$. The output $\pi$ is an extreme point of $\mathcal{P}_f$ generated with the order $\delta$ and sign vector $\sigma$.

\section{Poly-bimatroid Inequalities and the Full Description of $\mathcal{Q}_f$}
\label{sect:convex_hull}
In this section, we propose a class of valid inequalities for $\mathcal{Q}_f$, which we refer to as the \textit{poly-bimatroid inequalities}. Such inequalities are closely related to the extreme points of a bisubmodular polyhedron described in Section \ref{sect:poly_ineq}. 
Then we establish the full description of the convex hull $\mathcal{Q}_f$. Throughout our discussion, we refer to $\boldsymbol{-1}\leq x \leq \boldsymbol{1}$ as the \textit{trivial inequalities}. 

\noindent Our proposed class of inequalities is defined as follows.
\begin{definition}
A poly-bimatroid inequality is given by $z\geq \pi^\top x$ for any $\pi\in\mathcal{P}_f$. If $\pi$ is an extreme point of $\mathcal{P}_f$, then the corresponding poly-bimatroid inequality is called an extremal poly-bimatroid inequality. 
\end{definition}

 \begin{proposition}
\label{prop:1}
An inequality of the form $z\geq \pi^\top x$ is valid for $\mathcal{Q}_f$ if and only if it is a poly-bimatroid inequality. 
\end{proposition}
\begin{proof}
\noindent If $\pi \in\mathcal{P}_f$, then for any $x \in \{\pm 1, 0\}^n$, $\pi^\top x = \sum_{i\in S_1^x} \pi_i - \sum_{j\in S_2^x}\pi_j \leq f(S_1^x, S_2^x) = f(x) \leq z$. Conversely, suppose that  $\pi^\top x \leq z$ is valid for $\mathcal{Q}_f$. Notice that for any $x \in \{\pm 1, 0\}^n$, $(x, f(x))\in\mathcal{Q}_f$. Thus $\pi^\top x \leq f(x)$ for all $x \in \{\pm 1, 0\}^n$. This is equivalent to $\sum_{i\in S_1} \pi_i - \sum_{j\in S_2}\pi_j \leq f(S_1,S_2)$ for all $(S_1,S_2)\in 3^N$. 
\end{proof}

\noindent Note that this proposition is a generalization of Proposition 1 in \cite{atamturk2008polymatroids} for submodular functions  to the bisubmodular case. It follows from Proposition \ref{prop:1} that all  extremal poly-bimatroid inequalities are valid for $\mathcal{Q}_f$.

\noindent In the next two propositions, we establish some properties of the facets of $\mathcal{Q}_f$.

\begin{proposition}
\label{prop:ineq_form}
Any non-trivial facet-defining inequality $\pi^\top x \leq \alpha z + \pi_0$ for $\mathcal{Q}_f$ satisfies $\pi_0\geq 0$ and $\alpha=1$ up to scaling. 
\end{proposition} 
\begin{proof}
\noindent Since the ray $(\mathbf{0}, 1)$ is in $\mathcal{Q}_f$, $\alpha$ has to be non-negative. When $\alpha = 0$, for $\pi^\top x \leq \pi_0$ to be valid, $\pi_0 \geq \min_{x\in\{\pm 1\}^N} \pi^\top x$. The tightest inequality among such inequalities is when $\pi_0 = \min_{x\in\{\pm 1\}^N} \pi^\top x$, which is implied by the trivial inequalities.  Thus $\alpha >0$ and we can scale the inequality so that $\alpha = 1$. Given the assumption that $f(\emptyset, \emptyset) = 0$, $(\mathbf{0},0)\in\mathcal{Q}_f$, we have $\pi_0 \geq 0$. 
\end{proof}

\begin{proposition}
\label{prop:iff}
Let $f$ be a bisubmodular function with $f(\emptyset, \emptyset) = 0$. The inequalities of the type $\pi^\top x \leq z$ are facet-defining for $\mathcal{Q}_f$ if and only if they are extremal poly-bimatroid inequalities. 
\end{proposition}
\begin{proof}
\noindent First we show that the extremal poly-bimatroid inequalities are facet-defining. Let $e_i$ be a vector of dimension $n$ with all 0s but 1 in the $i$th entry. Consider the following $n+2$ points in $\mathcal{Q}_f$: $\mathbf{0}$, $\{(e_i^\top, f(\{i\},\emptyset))\}_{i=1}^n$ and $(-e_1^\top, f(\emptyset, \{1\})+\epsilon)$ where $\epsilon >0$ such that $f(\emptyset, \{1\}))+\epsilon \neq -f(\{1\},\emptyset)$. The non-zero points are linearly independent. Therefore, dim$(\mathcal{Q}_f)=n+1$. For each extremal poly-bimatroid inequality to be facet defining, we need $n+1$ affine independent points on its face. As mentioned in the discussion of Algorithm \ref{signedgreedy}, each extreme point $\pi\in\mathcal{P}_f$ has a consistent pair of ordering $\delta$ and sign vector $\sigma$. For $i=1,2,\dots, n$, let $\sigma^i$ be the same as $\sigma$ except 0s in the $\delta_i, \delta_{i+1},\dots,\delta_{n}$ entries. Note that $(\sigma^1,f(\sigma^1)),(\sigma^2, f(\sigma^2)),\dots,(\sigma^n, f(\sigma^n))$ along with $(\sigma, f(\sigma))$ lie on the face of $\pi^\top x \leq z$ and are affine independent. Therefore, the extremal poly-bimatroid inequalities are facet defining for $\mathcal{Q}_f$.  

\noindent Conversely, suppose $\bar{\pi}^\top x \leq z$ is facet-defining for $\mathcal{Q}_f$. Let $\Pi$ denote the set of extreme points  of $\mathcal{P}_f$ given by $\{\pi^1,\dots,\pi^{|\Pi|}\}$. By contradiction, we assume that $\bar{\pi}\notin \Pi$. If $\bar{\pi}\notin\mathcal{P}_f$ then there exists a disjoint pair of subsets $(S_1, S_2)$ such that $\bar{\pi}^\top x_{S_1,S_2} > f(S_1, S_2) \geq z$, making $\bar{\pi}^\top x \leq z$ invalid. On the other hand, if $\bar{\pi}\in\mathcal{P}_f \backslash \Pi$, then $\bar{\pi}= \sum_{i = 1}^{|\Pi|} \lambda_i \pi^i$, where $0\leq \lambda_i < 1$ and $\sum_{i=1}^{|\Pi|} \lambda_i = 1$. We have just shown that $\lambda_i(\pi^i)^\top x \leq \lambda_i z$ for $i = 1, 2, \dots, |\Pi|$ are facet-defining, and these inequalities imply $\bar{\pi}^\top x \leq z$. 
\end{proof}

\noindent Before we give our main result, we prove a useful lemma. 

\begin{lemma}\label{lem:f'}
Let  $f:3^N\rightarrow \mathbb{R}$ be a bisubmodular function with $f(\emptyset,\emptyset) = 0$. Then for any $\pi_0 > 0$, the function $f':3^N\rightarrow \mathbb{R}$ defined as 
$f'(S_1, S_2) = f(S_1, S_2) + \pi_0$ if $(S_1,S_2) \ne (\emptyset, \emptyset)$ and $f'(\emptyset, \emptyset) =0$ is bisubmodular. 
\end{lemma}

\begin{proof}
Let us consider any $(\emptyset,\emptyset) \neq (X_1, X_2)\in 3^N, (\emptyset, \emptyset) \neq (Y_1, Y_2)\in 3^N$. Notice that 
\begin{align*}
f'(X_1, X_2) + f'(Y_1, Y_2) & = f(X_1, X_2) + \pi_0 + f(Y_1, Y_2)+ \pi_0 \\
& \geq f(X_1\cap Y_1, X_2 \cap Y_2) + \pi_0 + f((X_1\cup Y_1 )\backslash (X_2 \cup Y_2), (X_2 \cup Y_2)\backslash (X_1\cup Y_1 ))+ \pi_0 \\
& \geq f'(X_1\cap Y_1, X_2 \cap Y_2) + f'((X_1\cup Y_1 )\backslash (X_2 \cup Y_2), (X_2 \cup Y_2)\backslash (X_1\cup Y_1 )).
\end{align*} The first inequality holds because $f$ is bisubmodular. The last inequality is not a strict equality because either $(X_1\cap Y_1, X_2 \cap Y_2)$ or $((X_1\cup Y_1 )\backslash (X_2 \cup Y_2), (X_2 \cup Y_2)\backslash (X_1\cup Y_1 ))$ can be $(\emptyset, \emptyset)$ even when both $(X_1, X_2)$ and $(Y_1, Y_2)$ are not $(\emptyset, \emptyset)$. For instance, $(X_1, X_2)=(\emptyset, \{i\})$ and $(Y_1, Y_2) = (\{i\}, \emptyset)$. Since $\pi_0 >0$, the last inequality holds. Without loss of generality, suppose $(X_1, X_2)=(\emptyset,\emptyset)$ while $(Y_1, Y_2)$ is any biset in $3^N$. In this case, observe that $(\emptyset\cap Y_1, \emptyset \cap Y_2) = (\emptyset, \emptyset)$ and $((\emptyset \cup Y_1 )\backslash (\emptyset \cup Y_2), (\emptyset \cup Y_2)\backslash (\emptyset \cup Y_1 )) = (Y_1, Y_2)$. It follows that 
\[
f'(\emptyset, \emptyset) + f'(Y_1, Y_2) = f'(\emptyset\cap Y_1, \emptyset \cap Y_2) + f'((\emptyset \cup Y_1 )\backslash (\emptyset \cup Y_2), (\emptyset \cup Y_2)\backslash (\emptyset \cup Y_1 )).
\] Therefore, $f'$ is bisubmodular. 
\end{proof}

\noindent Now we are ready to show that the extremal poly-bimatroid inequalities and the trivial inequalities give a complete linear description of $\mathcal{Q}_f$.

\begin{theorem}\label{thm:conv}
Suppose $f:3^N\rightarrow \mathbb{R}$ is a bisubmodular function with $f(\emptyset,\emptyset) = 0$. The convex hull of the epigraph of $f$, $\mathcal{Q}_f$, is described completely by the extremal poly-bimatroid inequalities $\pi^\top x \leq z$, and the trivial inequalities $\boldsymbol{-1}\leq x \leq \boldsymbol{1}$. 
\end{theorem}

\begin{proof}
\noindent By Proposition \ref{prop:ineq_form}, we know that the non-trivial facets of $\mathcal{Q}_f$ assume the form $\pi^\top x \leq z + \pi_0$ where $\pi_0 \geq 0$. For contradiction, suppose $\pi_0 > 0$. We observe that $\pi\notin \mathcal{P}_f$ because otherwise $\pi^\top x\leq z$ is valid and dominates $\pi^\top x\leq z +\pi_0$. Since $\pi^\top x \leq z + \pi_0$ is facet-defining for $\mathcal{Q}_f$, {$\pi^\top x \leq z$} is facet-defining for $\mathcal{Q}_{f'}$ where $f'(S_1, S_2) = f(S_1, S_2) + \pi_0$ if $(S_1,S_2) \ne (\emptyset, \emptyset)$ and $f'(\emptyset, \emptyset) =0$. {To see this, we observe that $\pi^\top x\leq z +\pi_0$ is not binding at $(x_{\emptyset,\emptyset}, f(\emptyset, \emptyset))$, and all the other points in the epigraph of $f$ are exactly the points in the epigraph of $f'$ shifted by $\pi_0$ in $z$.} From Lemma \ref{lem:f'}, the function $f'$ is bisubmodular.  By Proposition \ref{prop:iff}, $\pi$ is an extreme point in $\mathcal{P}_{f'}$, and it has a consistent order $\delta$ and a sign vector $\sigma$. With the signed greedy algorithm (Algorithm \ref{signedgreedy}), we can generate an extreme point $\pi^*\in\mathcal{Q}_f$ from the associated order $\delta$ and sign vector $\sigma$. Furthermore, $\pi = \pi^* + \pi_0{\sigma_{\delta(1)}} e_{\delta(1)}$. Thus $\pi^\top x \leq z + \pi_0$ is dominated by $\pi^{*\top} x \leq z$ and $e_{\delta(1)}^\top x \leq 1$ or $e_{\delta(1)}^\top x \geq -1$, contradicting the assumption that this inequality is a facet for $\mathcal{Q}_f$. Hence $\pi_0 = 0$. We conclude that the only facets for $\mathcal{Q}_f$ are the extremal poly-bimatroid inequalities and the trivial inequalities. 
\end{proof}

\noindent Propositions \ref{prop:ineq_form} and \ref{prop:iff}, and Theorem \ref{thm:conv} are generalizations of the results of \citet{atamturk2019submodular}  for submodular functions to the bisubmodular case.

\section{A Cutting Plane Algorithm for Bisubmodular Minimization}
\label{sect:cut_plane}
In this section, we consider constrained bisubmodular minimization problems in the form
\begin{subequations}
\label{eq:general}
\begin{alignat}{2}
\min & \quad f(S_1,S_2) \\
\text{s.t.} & \quad (S_1,S_2) \in \mathcal{S},
\end{alignat}
\end{subequations}
where $\mathcal{S}\subseteq 3^N$ captures additional restrictions such as cardinality or budget constraints. We can rewrite such problems as
\begin{subequations}
\label{eq:general_problem}
\begin{alignat}{2}
\min & \quad z  \\
\text{s.t.} & \quad (x,z)\in \mathcal{C} \label{subeq:genpoly}\\
& \quad x \in \mathcal{X}.
\end{alignat}
\end{subequations} The {polyhedral set $\mathcal{C}$ is defined by} the extremal poly-bimatroid inequalities which provide a piecewise linear representation of the objective function value. The set $\mathcal{X}$ encodes the constraints on the set  incidence vector $x$ defined by $\mathcal S$ in addition to the ternary restrictions.  

\noindent Note that there are exponentially many extremal poly-bimatroid inequalities in \eqref{subeq:genpoly}. Therefore, we propose Algorithm \ref{alg:general} to address Problem (\ref{eq:general_problem}). In this algorithm, we start with {a relaxed  set $\mathcal{C}$} that contains a subset of the extremal poly-bimatroid inequalities. 
Let  the optimality gap be measured by $(\text{UB}-\text{LB})/\text{UB}$, where UB is the upper bound  and LB is  the lower bound on the objective. 
While the optimality gap is greater than a given tolerance $\epsilon$, we solve this relaxed version of Problem (\ref{eq:general_problem}) to obtain $(\overline{x}, \overline{z})$. The current solution $\overline{z}$ is a lower bound for the optimal objective, and $f(\overline{x})$ is an upper bound. If $\overline{z}$ underestimates $f(\overline{x})$, then we add a most violated extremal poly-bimatroid inequality to $\mathcal{C}$ by solving a separation problem in the form of \eqref{eq:linear_bisub_max}  with the Generalized Greedy Algorithm provided in Algorithm \ref{alg:sepa}. We solve the updated relaxed problem in the next iteration and repeat.

 \vspace{0.3cm} 

\begin{algorithm}[H]
 \caption{Delayed Constraint Generation}
\label{alg:general}
\SetAlgoLined
\textbf{Input} initial $\mathcal{C}$, $\text{LB} = -\infty$, $\text{UB} = \infty$\;
 \While{ $(\text{UB}-\text{LB})/\text{UB}>\epsilon$ }{
  Solve Problem (\ref{eq:general_problem}) to get $(\overline{x}, \overline{z})$\;
  $\text{LB}\leftarrow\overline{z}$\;
     compute $f(\overline{x})$\;
   \If{$\overline{z}<f(\overline{x})$}{
    Use Algorithm  \ref{alg:sepa} to obtain a violated extremal poly-bimatroid inequality $z\geq \pi^\top x$. Add this cut to $\mathcal{C}$\;
     }
     \If{$\text{UB}>f(\overline{x})$}{
     $\text{UB}\leftarrow f(\overline{x})$\;
    Update the incumbent solution to $\overline{x}$ \;
     }
  }
  \textbf{Output} $\overline{z}$, $\overline{x}$. 
\end{algorithm}

\section{An Application: Robust Coupled Sensor Location}
\label{sect:computation}
Bisubmodularity has a wide range of applications including bicooreperative games \cite{bilbao2008survey}, coupled sensor placement, and coupled feature selection problems \cite{singh2012bisubmodular}. Even though most of these applications involve bisubmodular maximization, bisubmodular minimization appears as a subproblem in the adversarial or robust settings, as well as in the contexts such as obnoxious facility siting \cite{church1978locating,plastria1999undesirable}. \revised{In this section, we show an application of bisubmodular minimization in the context of a \emph{robust} variant of the coupled sensor placement problem  \cite{singh2012bisubmodular, ohsaka2015monotone}, which accounts for the worst-case outcomes due to the uncertainties in sensor deployment.  }

\noindent We first describe the coupled sensor placement problem. Let $N$ be a set of $n$ deployment locations. There are two types of sensors for different measurements, such as temperature and humidity. Every location can hold at most one sensor. Each biset $(S_1, S_2)\in 3^N$ represents the deployment of type-1 sensors at locations in $S_1$ and type-2 sensors at $S_2$. The effectiveness of a coupled sensor placement plan can be measured by entropy \cite{ohsaka2015monotone}. Suppose $X$ is a discrete random variable, and $\mathcal{X}$ is the set of all possible outcomes. Then the entropy of $X$ is defined by
\[H(X) = -\sum_{x\in \mathcal{X}} \mathbb{P}(X = x)\log \mathbb{P}(X = x). \]
Intuitively, the random variables with unpredictable outcomes have high entropy values. For example,  a fair coin has higher entropy than a biased coin because we have less information about the outcome of the fair coin. Let $X_{S_1,S_2}$ be a discrete random variable representing the observations collected from the sensor deployment $(S_1, S_2)$. The entropy of $X_{S_1,S_2}$ is given by $H(X_{S_1,S_2}) = -\sum_{x\in \mathcal{X}_{S_1,S_2}} \mathbb{P}(X_{S_1,S_2} = x)\log \mathbb{P}(X_{S_1,S_2} = x),$ where $\mathcal{X}_{S_1,S_2}$ is the set of observations at $(S_1, S_2)$. The objective function is $f(S_1, S_2) =  H(X_{S_1,S_2})$ for all $(S_1, S_2)\in 3^N$. It is desirable to place sensors at the locations where we are least sure about what we may observe. In other words, we would like to deploy sensors at the locations with high entropy to capture the most missing information about the entire environment. Therefore, the optimal sensor placement is usually found by maximizing function $f$. \citet{ohsaka2015monotone} show that the function $f$ is bisubmodular by verifying both Ando Conditions provided in Section \ref{sect:intro}.

\noindent Now consider a robust variant of the coupled sensor placement problem with two types of uncertainties. First, sensors of the wrong types may be installed at certain locations. Second, some sensors may fail due to software or hardware fault. We assume that at most $W$ sensors of the wrong type may be installed, and no fewer than $B'_1$ type-1 sensors and $B'_2$ type-2 sensors function properly. The goal is to determine a placement plan $(S_1,S_2)\in 3^N$ for $B_1$ type-1 sensors and $B_2$ type-2 sensors, such that the worst-case entropy is maximized.  To better illustrate this robust problem, we provide an example. 

\begin{example}
Suppose we plan to deploy two types of sensors over three locations. Type-1 sensors report temperatures in terms of "low" and "high". Type-2 sensors report humidities in terms of "humid" and "dry". Temperature and humidity data at these three locations is given in Table \ref{tab:ex}.  

\begin{table}[htb]
\footnotesize
\begin{center}
\caption{\footnotesize Temperature and humidity data collected at 3 sensor placement locations.}
\begin{tabular}{c|c|ccccccc}
\hline
 & location & day 1 & day 2 & day 3 & day 4 & day 5 & day 6 & day 7 \\
\hline
 temperature &  \multirow{2}{*}{1}  & low & low & low & high & high & high & low \\
humidity  & & humid & dry & dry & dry & dry & dry & humid \\ 
   \hline
   temperature &  \multirow{2}{*}{2}  & low & low& low & low & low & high & low \\
humidity  & & humid & humid & dry & dry & dry & humid & dry \\ 
   \hline
temperature &  \multirow{2}{*}{3}  & high & high & high & high & high & low & low  \\
humidity  & & humid & humid & dry & dry & humid & humid & humid \\ 
   \hline
\end{tabular}
\label{tab:ex}
\end{center}
\end{table}

\noindent Consider the placement plan $(\{1,3\},\{2\})$, by which we place temperature sensors at locations 1 and 3, and a humidity sensor at location 2. The probability of each outcome is
\begin{align*}
& p([\textrm{low,humid, high}])  = p([ \textrm{high, dry, high}]) = 2/7, \\
& p([ low, dry, high])  = p([ high, humid, low]) = p([ \textrm{low, dry, low}]) = 1/7. 
\end{align*} Then the entropy of this plan  is $ H(\{1,3\},\{2\})= - [2\times (2/7) \log_2 (2/7) + 3\times (1/7)\log_2 (1/7) ]  = 2.24.$

\noindent Suppose the parameters are set to $B_1 = 2, B_2 = 1, B'_1 = 1, B'_2 = 1$ and $W = 2$. Then the numbers of two types of sensors are 2 and 1 respectively. At most two sensors of the wrong types could be installed, and at least one installed sensor of each type stays intact. In this case, $(\{1,3\}, \{2\})$ is an optimal robust deployment plan. Its worst-case entropy is 1.38, when the temperature sensor placed at location 1 is broken and the humidity and temperature sensors at locations 2 and 3 are switched due to installation mistake. In other words, in the worst case, the working sensors are placed at $(\{2\},\{3\})$. The outcomes are [low, humid], [low, dry] and [high, humid], with probabilities $4/7$, $2/7$, and $1/7$, respectively. The entropy is $H(\{2\},\{3\})= - [(1/7) \log_2 (1/7) + (2/7)\log_2 (2/7) + (4/7)\log_2 (4/7)]  = 1.38. $
\end{example}

\noindent The robust coupled sensor placement problem can be written as a max-min optimization problem \eqref{eq:robust-entropy}.
\begingroup
\allowdisplaybreaks
\begin{subequations}
\label{eq:robust-entropy}
\begin{alignat}{2}
\max_{(S_1,S_2)\in 3^N} \min_{(T_1, T_2) \in 3^{S_1\cup S_2}} \hspace{0.2cm} & f(T_1, T_2)&& \\
\textrm{s.t.} \quad & |S_1| = B_1,  |S_2| = B_2, &&\\
& |T_1| \geq B'_1,  |T_2| \geq B'_2, && \\
& |T_1\cap S_2| + |T_2\cap S_1| \leq W.  && 
\end{alignat}
\end{subequations}
\endgroup
Here, $(S_1,S_2)$ represents the original deployment plan, and $(T_1, T_2)$ represents the locations of the deployed and functioning sensors. The inner problem is a constrained bisubmodular minimization problem. It contains multiple cardinality constraints and has a highly non-linear objective. Therefore, this inner problem does not have an equivalent compact mixed integer linear formulation that can be input to a solver, and we cannot apply common approaches such as converting the max-min problem into a maximization problem by duality. \revised{Approaches such as the integer L-shaped method \cite{laporte1993integer} can be used to solve this robust optimization problem. Such an algorithm iteratively determines---until an optimality condition is met---candidate coupled sensor locations,  $S_1$ and $S_2$, from the outer maximization problem. Given $S_1,S_2$, the inner minimization subproblem is formulated as  the bisubmodular minimization subproblem \eqref{eq:DCG}, which evaluates the worst-case entropy for this set of  coupled sensor locations. 
In this paper, we focus on solving this subproblem.
}
\begingroup
\allowdisplaybreaks
\begin{subequations}
\label{eq:DCG}
\begin{alignat}{2}
\min \hspace{0.2cm} & z  \\
\textrm{s.t.}\hspace{0.2cm}  & (x ,z) \in \mathcal{C}, && \label{subeq:linearapprox} \\
& x_i = y_i^1 -y_i^2, && \quad \text{ for all } i\in N, \label{subeq:biset}\\
& y_i^1 + y_i^2 \leq 1, &&\quad \text{ for all } i \in N, \label{subeq:disjoint}\\
& \sum_{i\in N} y_i^1 \geq B'_1, && \label{subeq:card1}\\
& \sum_{i\in N} y_i^2 \geq B'_2, && \label{subeq:card2}\\
& \revised{y_i^1, y_i^2 = 0,} &&\quad \revised{\text{ for all } i\in N\backslash (S_1\cup S_2), } \label{subeq:outer}\\
& \sum_{i\in S_1}y_i^2 + \sum_{i\in S_2}y_i^1 \leq W, &&\quad \text{ for all } i \in N, \label{subeq:switch}\\
& y_i^1, y_i^2 \in \{0,1\}, &&\quad  \text{ for all } i\in N. 
\end{alignat}
\end{subequations}
\endgroup
\noindent Here, the decision variables $x$ and $z$ are defined as in Problem (\ref{eq:general_problem})\revised{, such that $x$ is a ternary vector representing the locations of the working sensors, and $z$ is the auxiliary variable capturing the entropy.} Constraint \eqref{subeq:linearapprox} is the piecewise linear representation of the objective function value using the extremal poly-bimatroid inequalities. Constraint \eqref{subeq:biset} converts the ternary characteristic vector $x$ into the difference between two binary characteristic vectors $y^1$ and $y^2$. For every $i\in N$, $y_i^1=1$ when $i$ in assigned to the first argument of $f$, and $y_i^2 =1$ otherwise. Inequality \eqref{subeq:disjoint} ensures that the two arguments of $f$ are disjoint, and  inequalities \eqref{subeq:card1}-\eqref{subeq:card2} capture the cardinality constraints for the functioning sensors of both types. Constraint \eqref{subeq:outer} incorporates the solution from the outer-level maximization problem and restricts the feasible set in the inner-level minimization problem. Lastly, inequality \eqref{subeq:switch} is the cardinality upper bound of the number of sensors that may be misplaced. 

\section{Computational Experiments}
\label{sect:exp}
We conduct computational experiments on  \revised{the minimization of the entropy function motivated by} the robust coupled sensor placement problem. We use the Intel Berkeley Research lab dataset \cite{bodik2004intel}, which contains sensor readings including temperature and humidity at 54 sensor locations in the Intel Berkeley Research lab from February 28th to April 5th in 2004. The temperature data is discretized into three equal-width bins from the lowest temperature reading to the highest. The humidity measurements are discretized similarly into two equal-width bins.  Our computational  experiments are performed using two threads of a Linux server with Intel Haswell E5-2680 processor at 2.5GHz and 128GB of RAM using Python 3.6 and Gurobi Optimizer 7.5.1 in its default settings. The time limit for each instance is set to one hour. 

\noindent We randomly select $n\in\{5,10,\revised{20}\}$ out of 54 locations from the dataset for potential placement of sensors measuring temperature or humidity. We also randomly select $t\in\{10,20,50,100,\revised{500}\}$ time steps, and at which we collect pairs of discretized temperature and humidity data in the $n$ locations from the dataset. We set $B_1 = \floor{2n/5}$, $B_2 = \floor{n/2}$, $B'_1= \floor{4B_1/5}$, $B'_2=\floor{3B_2/5}$, and $W=\floor{3(B_1+B_2)/5}$. The parameters $B'_1$ and $B'_2$ are no higher than $B_1$ and $B_2$ respectively to account for the broken sensors, and $W$ is lower than the total number of sensors installed to account for the ones of wrong types. \revised{Ten sets of feasible initial coupled sensor placement, $S_1$ and $S_2$, are generated for each instance of the bisubmodular minimization problem.} The computational results are summarized in Table \ref{res:entropy}.  

\begin{table}[htb]
\small
\begin{center}
\caption{\revised{Computational statistics of the DCG algorithm for the bisubmodular minimization subproblem in robust coupled sensor placement. The statistics are averaged across 10 randomly generated instances.}}
\begin{tabular}{c|c|c|c|c}
\hline
$n$ & $t$ & time (s) &   \# poly-bimatroid cuts & \# nodes \\
\hline
 \multirow{4}{*}{5} & 10 & 0.003  & 2.1  & 2.0 \\
 & 20 & 0.012  & 10.2  & 25.9 \\
 & 50 & 0.011  & 4.9  & 6.4 \\
 & 100 & 0.035  & 11.4  & 28.9 \\
 & 500 & 0.096  &  8.0 & 20.7 \\
\hline
 \multirow{4}{*}{10} & 10 & 0.104  & 62.4  & 493.1\\
 & 20 &  0.238 & 97.1  & 1113.1 \\
 & 50 & 0.339  & 77.9  & 750.0 \\
 & 100 & 0.896  &  114.9 & 1024.1\\
 & 500 &  4.261 & 142.6  & 1162.3\\
\hline
 \multirow{4}{*}{20} & 10 & 0.405  &  103.8  & 2015.4\\
 & 20 &  1.121 & 167.5  & 2906.7\\
 & 50 &  88.371 & 1568.9  & 127030.7 \\
 & 100 &  154.587 & 2257.2  & 122514.8 \\
 & 500 & 277.127  & 1800.5  & 111079.2\\
 \hline
\end{tabular}
\label{res:entropy}
\end{center}
\end{table}

\noindent \revised{The first two columns of Table \ref{res:entropy} are the number of locations for initial placement, $n$, and the number of observations, $t$, respectively. Columns 3-5 summarize the average DCG statistics per instance of the bisubmodular minimization problem \eqref{eq:DCG}, including the runtime (in seconds), the number of poly-bimatroid inequalities added and the number of branch-and-bound nodes visited. Overall, the running time of the DCG algorithm increases as $n$ and $t$ increase. The number of branch-and-bound nodes visited and the number of poly-bimatroid cuts added display similar increasing trends. The DCG algorithm runs quickly in all instances. When $n \leq10$, every bisubmodular minimization problem is solved within five seconds. In the trials with $n=20$, DCG takes no more than five minutes to solve each problem instance on average. } 

\section{Concluding Remarks}
\label{sect:conclusion}
\noindent In this paper, we introduce the poly-bimatroid inequalities which are valid for the convex hull of the epigraph of any bisubmodular function. We further prove that the extremal poly-bimatroid inequalities, along with the trivial inequalities, fully describe the convex hull. The delayed cut generation algorithm we propose based on such cuts \revised{is effective in solving the subproblems associated with the} highly non-linear robust entropy optimization problem which does not admit an equivalent compact mixed-integer linear formulation. Motivated by the effectiveness of cutting plane approaches for submodular maximization (e.g., \cite{NW81,Ahmed2011,Yu2017,Wu2017,Wu2018maxinf}), in a follow-up work \cite{yu2020exact}, we conduct a polyhedral study of  bisubmodular maximization. 

\section*{Acknowledgements}
We thank the reviewer for comments that improved the paper. This research was supported in part through the computational resources and staff contributions provided for the Quest high performance computing facility at Northwestern University which is jointly supported by the Office of the Provost, the Office for Research, and Northwestern University Information Technology.

\bibliography{bibliography}{}
\bibliographystyle{apalike}
 
\end{document}